\documentclass[a4paper,12pt,reqno]{amsart}
\usepackage[T1]{fontenc}
\usepackage{lmodern}
\usepackage[utf8]{inputenc}

\usepackage[left=2.2cm,right=2.2cm,top=2.2cm,bottom=2.7cm]{geometry}
\usepackage{amsthm,amsmath,amsfonts,amssymb,mathtools,extarrows}
\usepackage[bbgreekl]{mathbbol}
	\DeclareSymbolFontAlphabet{\mathbb}{AMSb}
	\DeclareSymbolFontAlphabet{\mathbbl}{bbold}
\usepackage[numbers]{natbib}
\usepackage[
	colorlinks,
	hyperfootnotes=true,
	hyperindex,
	linkcolor=blue,
	]{hyperref}
\usepackage{xcolor} % Nur noetig fuer Kommentar-Funktion
\allowdisplaybreaks

\numberwithin{equation}{section}

\theoremstyle{plain}
\newtheorem{theorem}{Theorem}[section]
\newtheorem{lemma}[theorem]{Lemma}

\theoremstyle{definition}

\theoremstyle{remark}
\newtheorem{remark}[theorem]{Remark}

\newcommand{\N}{\mathbb{N}}
\newcommand{\Zz}{\mathbb{Z}}
\newcommand{\R}{\mathbb{R}}

\newcommand{\Pp}{\mathbb{P}}
\newcommand{\Un}{\mathbbl{1}}
\newcommand{\D}{\mathbb{D}}
\newcommand{\E}{\mathbb{E}}
\newcommand{\Var}{\operatorname{Var}}

\newcommand{\Vol}{\operatorname{Vol}}

\newcommand{\ensemble}[1]{ \left\lbrace #1 \right\rbrace } 
\newcommand{\prth}[1]{\!\left( #1 \right) }
\newcommand{\abs}[1]{\left| #1 \right|}  
\newcommand{\norm}[1]{\left\Vert #1 \right\Vert}  

\newcommand{\Unens}[1]{ \Un_{ \ensemble{#1} } }

\title[Non-uniform Berry--Esseen bounds]{Non-uniform Berry--Esseen bounds \\ for Gaussian, Poisson and Rademacher processes}
\author[M. Butzek]{Marius Butzek}
\address{Faculty of Mathematics, Ruhr University Bochum, Germany.}
\email{marius.butzek@rub.de}
\author[P. Eichelsbacher]{Peter Eichelsbacher}
\address{Faculty of Mathematics, Ruhr University Bochum, Germany.}
\email{peter.eichelsbacher@rub.de}
\begin{document}
\begin{abstract}
In this paper we obtain non-uniform Berry-Esseen bounds for normal approximations by the Malliavin-Stein method.
The techniques rely on a detailed analysis of the solutions of Stein's equations and will be applied to functionals of a Gaussian process like
multiple Wiener-It\^o integrals, to Poisson functionals as well as to the Rademacher chaos expansion. Second-order Poincar\'e inequalities
for normal approximation of these functionals are connected with non-uniform bounds as well. As applications, elements living inside a fixed Wiener
chaos associated with an isonormal Gaussian process, like the discretized version of the quadratic variation of a fractional Brownian motion, are considered. 
Moreover we consider subgraph counts in random geometric graphs as an example of Poisson $U$-statistics, as well as subgraph counts in the Erd\H{o}s-R\'enyi random graph and infinite weighted 2-runs as examples of functionals of Rademacher variables.
\end{abstract}	

\maketitle
\section{Introduction}

The classical CLT is stated as follows: We consider a sequence $(X_k)_{k \in \N}$ of independent and identically distributed (i.i.d.) random variables with expectation $\mu := \mathbb{E}[X_1] < \infty$ and variance $\sigma^2 := \Var(X_1) < \infty$, for $n \in \N$ the $n$-th partial sum $S_n := X_1 + \ldots + X_n$ as well as the standardized $n$-th partial sum $W_n := (S_n - n \mu)/ \sqrt{n \sigma^2}$, then
\begin{align}\label{CLT_classic}
	W_n \overset{d}{\longrightarrow} Z \sim \mathcal{N}(0,1),
\end{align}
which means that $W_n$ converges in distribution to a standard-normal distributed random variable $Z$ as $n$ tends to infinity. The distribution function of a \textit{normal distributed} random variable $X \sim \mathcal{N}(\mu, \sigma^2)$, also known as \textit{Gaussian distribution}, is given by
\begin{align*}
\mathbb{P}(X \leq x) = \frac{1}{\sqrt{2 \pi \sigma^2}} \int_{-\infty}^x e^{\frac{(t - \mu)^2}{2 \sigma^2}} dt,
\end{align*}
we call the case $Z \sim \mathcal{N}(0,1)$ \textit{standard-normal} and define $\Phi(x) := \mathbb{P}(Z \leq x)$. 
\textsc{A. C. Berry} and \textsc{C.--G. Esseen} proved that under the assumption, that the third absolute moments of $X_1, \ldots, X_n$ are finite, one has the following bound, see \cite{B41} and \cite{E42},
\begin{align}\label{Berry_Esseen_classic}
	\sup_{x \in \R} \abs{ \mathbb{P}(W_n \leq x) - \mathbb{P}(Z \leq x) } & \leq \frac{C \cdot \mathbb{E}\abs{X_1}^3}{\sqrt{n}},
\end{align}
where $C$ is a constant. The left-hand side of \eqref{Berry_Esseen_classic} is essentially the difference of the distribution functions of $W_n$ and $Z$ and is called Kolmogorov distance. Note, that we have to distinguish \textit{uniform} and \textit{non-uniform} Berry--Esseen bounds. While uniform bounds as \eqref{Berry_Esseen_classic} have a supremum, here over all $x \in \R$, this is not the case for non-uniform bounds. As a consequence the right-hand side of such bounds has an additional prefactor depending on our real variable $x$, e.g. for independent and not necessarily identically distributed random variables is given as follows:
\begin{align}\label{NU_Bikelis}
\abs{\Pp(W_n \leq x) - \mathbb{P}(Z \leq x)} \leq C \sum_{i=1}^n  \frac{\mathbb{E}\abs{X_i}^3 }{1 + \abs{x}^3}.
\end{align}
Over the years a lot of primarily uniform, but also non-uniform Berry--Esseen results were developed for various applications.
The first bounds of this type came from Esseen himself in 1945 for independent and identically distributed (iid) random variables with finite third moments. They were improved by Nagaev in 1965 and generalized by Bikelis in 1966 for independent and not necessarily identically distributed random variables, which is \eqref{NU_Bikelis}.

In 2001, Chen and Shao \cite{CS01} generalized \eqref{NU_Bikelis} and proved their bound without assuming the existence of third moments, and truncating their random variables at 1:
\begin{align}\label{NU_CS}
\abs{\Pp(W \leq x) - \Phi(x)} \leq C \sum_{i=1}^n \prth{ \frac{\mathbb{E}[X_i^2] \Unens{\abs{X_i} > 1 + \abs{x}}}{(1 + \abs{x})^2} + \frac{\mathbb{E}\abs{X_i}^3 \Unens{\abs{X_i} \leq 1 + \abs{x}}}{(1 + \abs{x})^3}   }.
\end{align}
A few years later Chen und Shao \cite{CS04} established a similar result under local dependence. They obtained both of their results by a combination of Stein's method and the concentration inequality approach. Stein's method is a very useful tool by itself to bound distances of probability distributions and goes back to Charles Stein in 1972 with normal approximation as its original and most famous application.
\\ In the following years, continuations of the work of Chen and Shao can be found in \cite{BC04} for translated Poisson approximation and \cite{CS07} for nonlinear statistics. Moreover papers, which aimed for generalization or improvement of constants as \cite{P07}, \cite{S20}, \cite{KP12} and \cite{KP22}. 
\\ The specific starting point of our paper is \cite{LLWC21}, where Liu et al. showed non-uniform Berry--Esseen bounds for normal and nonnormal approximations by unbounded exchangeable pairs $(W,W')$. They referred to a corresponding uniform bound in \cite{SZ19} and proved their main result without concentration inequalities. Recently their work was generalized in \cite{SS23} for the normal approximation case under the additional assumption of $\mathbb{E}\abs{W-W'}^{2r}$ being of certain order.
\\ When we studied the proof of the main result in \cite{LLWC21}, our observation was the following: The non-uniform bound consists almost of the same terms, which were constructed by the theory of exchangeable pairs, as the uniform bound in \cite{SZ19}, but with a prefactor depending on $x \in \R$. The reason for that is a strict separation between the mentioned terms and the fragments of the solutions of the Stein-equation of the corresponding exchangeable pair. Most of the proof depends on the solution of the Stein-equation, which has to be bounded precisely to lead to the desired prefactor, and not on the exchangeable pair itself. In \cite{CS04},
we will find a simple but important bound on the derivative of the solution of Stein's equation, see \cite[Lemma 5.1]{CS01}. The proof of this Lemma was generalized in
\cite{LLWC21} for non-Gaussian approximations. 

These observations motivates our attempt to generalize the argumentation in \cite{LLWC21} for other applications and techniques related to Stein's method. We consider normal approximation of $L^2$-functionals of centered isonormal Gaussian processes, of $L^2$-functionals of certain Poisson processes as well as of $L^2$-functionals over infinitely many Rademacher random variables. The remaining parts of the paper are organized as follows. In Section 2 we present the preliminaries on Stein's method and the Malliavin-Stein method,
introducing the $L^2$-functionals of Gaussian and Poisson processes as well as of Rademacher random variables and certain operators. We also 
state our non-uniform Berry-Esseen bounds. The proofs will be presented in Section 3. In Section 4 we collect non-uniform Berry-Esseen bounds
in the context of second order Poincar\'e inequalities. Section 5 lists some applications.

The results of the paper are mostly parts of the PhD Thesis of Marius Butzek \cite{B24}, defended on 28th of May 2024. Very recently, in \cite{DVT24}
the authors consider non-uniform bounds via Malliavin-Stein and consider exponential functionals of Brownian motion. 
\bigskip

\section{Preliminaries}
\subsection{Stein's method}
We want to illustrate the main ideas of Stein's method for normal approximation and refer to \cite{BC05} and \cite{CGS11} for an extensive and nice overview. We start with the important characterization \cite[Lemma~2.1]{BC05}, also known as Stein-Lemma,
\begin{align}\label{Steinlemma}
	Z \sim \mathcal{N}(0,1) \Leftrightarrow \mathbb{E}[f^\prime(Z) - Z f(Z)] = 0
\end{align}
for all continuous differentiable $f$ such that the appearing expectations exist. So if a random variable is in some sense close to $\mathcal{N}(0,1)$, it is likely that the right hand side of \eqref{Steinlemma} is close to 0. This motivates the Stein-equation, written in the case of Kolmogorov distance, namely
\begin{align}\label{Steinequation}
	f_z^\prime(F) - F f_z(F) = \Unens{F \leq z} - \Phi(z),
\end{align}
respectively
\begin{align*}
	\mathbb{E}[f_z^\prime(F) - F f_z(F)] = \Pp(F \leq z) - \Phi(z).
\end{align*}
The solution to this equation is given by
\begin{align*}
	f_z(w) = \begin{cases}
		\frac{\Phi(w) (1-\Phi(z))}{p(w)} & w \leq z, \\
		\frac{\Phi(z) (1-\Phi(w))}{p(w)} & \, w > z,
	\end{cases}
\end{align*}
where $p(w) = e^{-w^2/2} / \sqrt{2 \pi}$ is the density of $\mathcal{N}(0,1)$. Especially it follows that
\begin{equation} \label{first}
f'_z(w) = (1-\Phi(z)) \bigl(1 + \sqrt{2 \pi} w e^{w^2/2}  \Phi(w) \bigr), \,\, w \leq z.
\end{equation}
The following bounds are well known:
$$
	\abs{f_z^\prime(w)} \leq 1,  
	\,\, \abs{w  f_z(w)}  \leq 1 
$$
 for all $w \in \R$ and so in particular for the sup norm $\| \cdot \|_{\infty}$.
\bigskip
 
\subsection{Malliavin--Stein Method}
\noindent The idea of Malliavin--Stein method was developed by Nourdin and Peccati, and combines Stein's method with Malliavin calculus. In many cases it is applied to so called Rademacher-, Poisson- and Gaussian-functionals. Historically Malliavin calculus was applied first to Gaussian- \cite{NP12}, then to Poisson- \cite{PSTU10} and then to Rademacher-functionals \cite{NPR10} --- in fact these are all since for other distributions the chaos representations of the corresponding functionals do not exist. We want to summarize the setting and the operators that will appear in our results. For further details, results  and information related to the topic see \cite{Webpage}.

\subsubsection{The Gaussian case} \label{subGauss}

Let $\mathcal H$ be a real separable Hilbert space with scalar inner product denoted by $\langle \cdot, \cdot \rangle_{\mathcal H}$, and for $q \geq 1$, let ${\mathcal H}^{\otimes q}$ (resp. ${\mathcal H}^{\odot q}$) be the $q$th tensor product (resp. $q$th symmetric tensor product) of $\mathcal H$. We denote by $X= \{ X(h): h \in \mathcal H \}$ an {\it isonormal Gaussian process} over $\mathcal H$, which is a centered Gaussian family such that for every $h,g \in \mathcal H$,
$$
\E (X(h) X(g)) = \langle h,g \rangle_{\mathcal H}.
$$
Let $L^2(X)= L^2(\Omega, \sigma(X),P)$ be the space of square integrable functionals of $X$. It is known that any $F \in L^2(X)$ admits the chaotic expansion
\begin{equation} \label{chaos}
F = \sum_{q=0}^{\infty} I_q(f_q),
\end{equation}
where $I_0(f_0) := \E(F)$, the series converges in $L^2$ and the symmetric kernels $f_q \in \mathcal H^{\odot q}, q \geq 1$, are uniquely determined by $F$.
For every $f \in \mathcal H^{\odot q}$, the random variable $I_q(f)$ coincides with the {\it multiple Wiener-It\^o integral} of order $q$ of $f$ with respect to $X$.
For details see \cite{N2006}. Remark that $I_q(f_q)$  has finite moments of all order. Let us denote by $J_q$ the orthogonal projection operator on the $q$th Wiener chaos associated with $X$, so that, if $F \in L^2(X)$
as in \eqref{chaos}, then $J_q F =I_q(f_q)$ for every $q \geq 0$. 

Let $\mathcal S$ be the set of all smooth cylindrical random variables of the form
$$
F = g(X(\phi_1), \ldots, X(\phi_n)),
$$
where $n \geq 1$, $g: \R^n \to \R$ is a smooth function with compact support and $\phi_i \in \mathcal H$, $i=1, \ldots, n$. The {\it Malliavin derivative}
of $F$ with respect to $X$ is the element of $L^2(X)$ defined as
$$
DF = \sum_{i=1}^n \frac{\partial g}{\partial x_i} (X(\phi_1), \ldots, X(\phi_n)) \, \phi_i.
$$
By iteration one can define the $m$th derivative $D^mF$ for every $m \geq 2$. As usual, for $m \geq 1$, $\D^{m,2}$ denotes the closure of $\mathcal S$
with respect to the norm $\| \cdot \|_{m,2}$, defined by the relation
$$
\| F \|_{m,2} = \E (F^2) + \sum_{i=1}^m \E \bigl( \| D^i F \|_{{\mathcal H}^{\otimes i}} \bigr).
$$
The operator $L$, acting on square integrable random variables of the type \eqref{chaos}, is defined through the projection operators $(J_q)_{q \geq 0}$ as
$L = \sum_{q=0}^{\infty} - q \, J_q$, and is called the infinitesimal generator of the {\it Ornstein-Uhlenbeck semigroup}. The domain $\text{Dom}L$ is the space
$\D^{2,2}$. A random variable $F$ as in \eqref{chaos} is in $\D^{1,2}$ (resp. $\D^{2,2}$) if, and only if,
$$
\sum_{q=1}^{\infty} q \| f_q \|_{{\mathcal H}^{\odot q}}^2 < \infty \quad \biggl( \text{resp.} \,\, \sum_{q=1}^{\infty} q^2 \| f_q \|_{{\mathcal H}^{\odot q}}^2 < \infty \biggr),
$$
and also $\E \bigl( \| DF \|_{\mathcal H}^2 \bigr) = \sum_{q=1}^{\infty} q \| f_q \|_{{\mathcal H}^{\odot q}}^2$. We also define the operator $L^{-1}$,
which is the {\it pseudo-inverse} of $L$, as follows: for every $F \in L^2(X)$ with zero mean, we set $L^{-1} F = \sum_{q \geq 1} - \frac 1q J_q(F)$. Note that $L^{-1}$
is an operator with values in $\D^{2,2}$. Now the famous result in \cite{NP09} reads as following:

\begin{theorem}[Gaussian approximation via the Malliavin-Stein method, \cite{NP09}] \label{T1}
Let $Z \sim {\mathcal N}(0,1)$ and let $F \in \D^{1,2}$ be such that $\E(F)=0$ and $\E(F^2)=1$. Then the following bounds are in order:
$$
d_W(F,Z) \leq \E \bigl( (1- \langle DF, - DL^{-1} F \rangle_{\mathcal H} )^2 \bigr)^{1/2},
$$
where $d_W(F;Z)$ denotes the Wasserstein-distance between $F$ and $Z$. If, in addition, the law of $F$ is absolutely continuous with respect to the Lebesgue
measure, one has that
$$
\sup_{x \in \R} | P(F \leq x) - P(Z \leq x)| \leq \E \bigl( (1- \langle DF, - DL^{-1} F \rangle_{\mathcal H} )^2 \bigr)^{1/2}.
$$
If in addition $F \in \D^{1,4}$, then  $\langle DF, - DL^{-1} F \rangle_{\mathcal H} $ is square-integrable and 
$$ \E \big| 1- \langle DF, - DL^{-1} F \rangle_{\mathcal H}  \big| \leq \bigl( \Var \bigl(  1- \langle DF, - DL^{-1} F \rangle_{\mathcal H} \bigr) \bigr)^{1/2}.
$$
\end{theorem}

\noindent
The famous fourth-moment theorem is saying, that if $F=I_q(f)$ with $\E(F^2)=1$, then 
$$
\sup_{x \in \R} | P(F \leq x) - P(Z \leq x)| \leq \sqrt{ \frac{q-1}{3q} (\E(F^4)-3)},
$$
see \cite{NP12}, Theorem 5.2.6.
\bigskip

\noindent
Our result is:

\begin{theorem}[non-uniform Berry-Esseen bound on Wiener space] \label{Gaussianmain}
Let $Z \sim {\mathcal N}(0,1)$ and let $F \in \D^{1,2}$ be such that $\E(F)=0$ and $\E(F^2)=1$. 
If $\E(F^{2k}) < C$ for fixed $k \in \N$, and if in addition, the law of $F$ is absolutely continuous with respect to the Lebesgue
measure, one has that for $x \in \R$
$$
| P(F \leq x) - P(Z \leq x)| \leq \frac{C}{(1 + |x|)^k} \,\, \E \bigl( (1- \langle DF, - DL^{-1} F \rangle_{\mathcal H} )^2 \bigr)^{1/2}.
$$
Moreover, let $q \geq 2$ and let $F = I_q(f)$ with $\E(F^2)=1$, we obtain for $x \in \R$
$$
| P(I_q(f) \leq x) - P(Z \leq x)| \leq \frac{C}{(1 + |x|)^3} \,\, \sqrt{ \frac{q-1}{3q} (\E(F^4)-3)},
$$
because $\E (I_q(f)^6) < C$ is fulfilled.
\end{theorem}

\begin{remark}
The proof of Theorem \ref{Gaussianmain} will show that the prefactor $ \frac{C}{(1 + |x|)^k}$ can even be improved to
$$
 C e^{-z^2/4} + \sqrt{ P(F > z/2)}, \,\, z >0.
 $$
 Whenever there is a concentration inequality of exponential type available, this bound is even of exponential type. For any sequence of Wiener-It\^o integrals
 $I_q(f_n)$, these type of result is available, see Proposition 3 and Theorem 5 in \cite{STx}. It exists constants $c_n=c_n(q,f_n)$ such that
 \begin{equation} \label{conGauss}
 \Pp \bigl( | I_q(f_n) | \geq z \bigr) \leq 2 \exp \biggl( - \frac 14 \min \bigg\{ \frac{z^2}{2^{q/2}}, ( c_n(q,f_n) z )^{2/q} \bigg\} \biggr).
 \end{equation}
 \end{remark}
 The definition of $c_n$ will be given in the examples in Section \ref{Examples}. It is given as a formula in $q$ as well as in terms of  the $L^2$-norms of the contraction operators of $f_n$. 
\bigskip

\subsubsection{The Poisson case} \label{subPoisson} 
Let $(X,\mathcal{X})$ be a standard Borel space with a $\sigma$-finite measure $\mu$, which
is defined on an underlying probability space $(\Omega, \mathcal F, \Pp)$. By $\eta$ we denote a \textit{Poisson (random) measure}, also known as \textit{Poisson point process}, on $X$ with \textit{control} $\mu$.  Define $\mathcal{X}_0 = \{B \in \mathcal{X} : \mu(B) < \infty \}$ such that $\eta = \{\eta(B) : B \in \mathcal{X}_0 \}$ is a collection of random variables with the following properties:
\begin{itemize}
	\item  $\eta(B)$ is Poisson distributed with parameter $\mu(B)$ for each $B \in \mathcal{X}_0$.
	\item If $B_1,...,B_n \in \mathcal{X}_0$ are disjoint sets, the random variables $\eta(B_1),...,\eta(B_n)$ are independent, for all $n \in \N$.
\end{itemize}
Denote by $\Pp_\eta$ the distribution of $\eta$ and, if needed, by $\hat{\eta}$ the \textit{centered Poisson measure}
\begin{align*}
	\hat{\eta}(B) = \eta(B) - \mathbb{E}[\eta(B)] = \eta(B) - \mu(B).
\end{align*}
Last, we use the notations $L^2(\mu^n)$ and $L^2(\Pp_\eta)$ for the space of square-integrable functions with respect to $\mu^n$ respectively the space of square-integrable functionals with respect to $\Pp_\eta$.
\\ We are interested in functionals $F = F(\eta) \in L^2(\Pp_\eta)$ and it is a crucial feature of $\eta$ that any such $F$ can be written as
\begin{align*}
	F = \mathbb{E}[F] + \sum_{n=1}^\infty I_n(f_n),
\end{align*}
where $I_n$ is the \textit{$n$-fold Wiener-It$\hat{o}$ integral}, also known as \textit{Poisson multiple integral}, with respect to $\hat{\eta}$ and $(f_n)_{n \in \N}$ is a unique sequence of symmetric functions in $L^2(\mu^n)$, depending on $F$. We refer to section 3 in \cite{LP11} for formal details.
For such functionals we define the \textit{difference operator} $D_xF$ of $F$ at $x \in X$ by putting
\begin{align*}
	D_xF(\eta) & :=  F(\eta + \delta_x) - F(\eta), \\
	DF & : x \mapsto D_xF,
\end{align*}
also known as the \textit{add-one-cost operator} since it measures the effect on F of adding the point $x \in X$ to $\eta$. 
We obtain
\begin{align*}
	F \in {\hat{\D}}^{1,2} := Dom(D) = \left \{F \in L^2(\Pp_\eta) \bigg \vert \, \sum_{n=1}^\infty n \cdot n! \norm{f_n}_n^2 < \infty \right \},
\end{align*}
with $\norm{.}_n$  the norm in $L^2(\mu^n)$.
Supposing $F \in {\hat{\D}}^{1,2}$  such that $\sum_{n=1}^\infty n^2 \cdot n! \norm{f_n}_n^2 < \infty$, we define the \textit{Ornstein--Uhlenbeck operator} $L$ and the \textit{pseudo-inverse Ornstein--Uhlenbeck operator} $L^{-1}$ as $LF := \sum_{n=1}^\infty -n I_n(f_n)$ and $L^{-1}F  := \sum_{n=1}^\infty - \frac{1}{n}I_n(f_n)$,
where in the definition of $L^{-1} F$ we assume $F \in L^2(\Pp_\eta)$ with $\E(F)=0$. Finally, if $z \mapsto h(z)$ is a random function on $\mathcal X$
with chaos expansion $h(z) = z + \sum_{n=1}^{\infty} I_n(h_n(z, \cdot))$ with symmetric functions $h_n(z, \cdot) \in L^2(\mu^n)$ such that
$\sum_{n=0}^{\infty} (n+1)! \, \|h_n\|^2 < \infty$
(let us write $h \in \text{Dom}(\delta)$ if this is satisfied), the {\it Skorohod integral} $\delta(h)$ of $h$ is defined as
$$
\delta(h) = \sum_{n=0}^{\infty} I_{n+1}(\tilde{h}_n),
$$
where $\tilde{h}_n$ is the canonical symmetrization of $h_n$ as a function of $n+1$ variables. Sometimes $\delta$ is called the \textit{divergence operator}.
Interesting enough there are nice relationships between the operators $D$, $\delta$ and $L$, see \cite[Lemma~2.1]{ET14}. It holds that $F \in \text{Dom}(L)$ if and only if $F \in \text{Dom}(D)$ and $DF \in \text{Dom}(\delta)$, in which case $\delta(DF) = -LF$. The integration by parts formula is the key of the Malliavin-Stein method
and reads $\E(F \delta(h)) = \E \langle DF, h \rangle$ for every $F \in \text{Dom}(D)$ and $h \in \text{Dom}(\delta)$.

The Gaussian-approximation for $L^2$-functionals of Poisson processes via Malliavin-Stein started in \cite{PSTU10} and \cite{ET14}.
A recent result from \cite{LRPY20} seem to obtain a nearly optimal bound:

\begin{theorem}[Berry--Esseen bound for Poisson functionals, \cite{LRPY20}]\label{T2}
Let $F \in {\hat{\D}}^{1,2}$ with $\mathbb{E}[F] = 0, \Var(F) = 1$.  
Further
\begin{align*}
	\mathbb{E} \int_X \int_X \left[D_y \left(D_xF \abs{D_x L^{-1} F}\right)\right]^2 \mu^2(dx,dy) < \infty,
\end{align*}
\begin{align*}
F f_z(F) \in {\hat{\D}}^{1,2} \quad \forall z \in \R.
\end{align*}
Then
	\begin{align*}
\sup_{z \in \R} \abs{\Pp(F \leq z) - \Phi(z)} \leq \mathbb{E} | 1 - \langle DF , - DL^{-1} F \rangle | +  2 \mathbb{E} \bigl[ \abs{ \delta (  DF  | D L^{-1} F |  )  } \bigr].
	\end{align*}
\end{theorem}
\bigskip

\noindent 
Our result is:

\begin{theorem}[Non-uniform Berry-Esseen bound for Poisson functionals] \label{Poissonmain}
Let $F \in {\hat{\D}}^{1,2}$ with $\mathbb{E}[F] = 0, \Var(F) = 1$ and $\mathbb{E}[F^{2k}] < \infty$ for fixed $k \in \N$. 
Further
\begin{align*}
	\mathbb{E} \int_X \int_X \left[D_y \left(D_xF \abs{D_x L^{-1} F}\right)\right]^2 \mu^2(dx,dy) < \infty,
\end{align*}
\begin{align*}
F f_z(F) \in {\hat{\D}}^{1,2} \quad \forall z \in \R.
\end{align*}
Then, for any $z \in \R$,
	\begin{align*}
\abs{\Pp(F \leq z) - \Phi(z)} \leq \frac{C}{(1 + \abs{z})^k} \prth{\prth{\mathbb{E} \prth{ 1 - \langle DF , - DL^{-1} F \rangle }^2 }^{1/2} + \prth{\mathbb{E} \prth{  \delta\prth{  DF \abs{D L^{-1} F} } }^2 }^{1/2} },
	\end{align*}
and $C$ is a constant depending on $k \in \N$.
\end{theorem}

\begin{remark}
The proof of Theorem \ref{Poissonmain} will show that the prefactor $ \frac{C}{(1 + |x|)^k}$ can even be improved to
$$
 C e^{-z^2/4} + \sqrt{ P(F > z/2)}, \,\, z>0.
 $$
 Whenever there is a concentration inequality of exponential type available, this bound is even of exponential type. For any fixed Wiener-It\^o integral
 $I_q(f)$, these type of result is available, see Theorems 3.1 in \cite{ST23}. Under certain restrictions to all cumulants (measured by a sequence
 $(\alpha_n)$ in \cite[(3.1)]{ST23}) of order
 $m \geq 3$ simultaneously, there exists a sequence $c_n=c_n(q, f, \alpha_n)$ such that
 \begin{equation} \label{conPoisson}
 \Pp \bigl( | I_q(f) | \geq z \bigr) \leq 2 \exp \biggl( - \frac 14 \min \bigg\{ \frac{z^2}{2^{q}}, ( c_n(q, f, \alpha_n) z )^{1/q} \bigg\} \biggr).
 \end{equation}
 The definition of $c_n$ will be given in the examples in Section \ref{Examples}. The sequence $(\alpha_n)_n$ is a certain bound to all cumulants of order
 $m \geq 3$ of $F$ with respect to Poisson point processes $(\eta_n)_n$, see \cite[Condition (3.1)]{ST23}.
 %It is given as a formula in $q$ as well as by the $L^2$-norm of the contraction operators of $f_n$.
 For so called Poisson $U$-statistics, a similar concentration inequality was proved in \cite[Theorem 3.2]{ST23}.
 \end{remark}
\bigskip	

\subsubsection{The Rademacher case} \label{subRademacher}
\noindent We start with $X = (X_k)_{k \in \N}$, a sequence of Rademacher random variables, e.g. $\forall k \in \N$:
\begin{align*}
\Pp(X_k = 1) & = p_k \in (0,1) \\
\Pp(X_k = -1) & = q_k = 1 - p_k,
\end{align*}
and, if needed, the standardized random variable
\begin{align*}
Y_k = \frac{X_k - p_k + q_k}{2 \sqrt{q_k p_k}}.
\end{align*}
We are interested in random variables $F \in L^2(\Omega, \sigma(X), \Pp)$. We consider the Hilbert space  $\mathcal H$ of real square-summable sequences.
For $p \in \N$, we denote by $\mathcal H^{\otimes p}$ ($\mathcal H^{\odot p}$ respectively) the $p$th tensor product (symmetric tensor product resp.) of $\mathcal H$.
Let $\Delta_p := \{ (i_1, \ldots, i_p) \in \N^p: i_j \not= i_k \, \text{for} \, j \not= k \}$. Now consider $\mathcal H_0^{\odot p} = \{ f \in \mathcal H^{\odot p}: f = \tilde{f} 1_{\Delta_p} \}$ for $p \in \N$, where $\tilde{f}$ stands for the canonical symmetrization of $f$.

For any $F \in  L^2(\Omega, \sigma(X), \Pp)$, there is a unique sequence of kernels $f_n \in \mathcal H_0^{\odot n}$, $n \in \N$, such that
\begin{equation} \label{chaos3}
F = \E(F) + \sum_{n=1}^{\infty} J_n(f_n) \quad \text{in} \,\, L^(\Omega).
\end{equation}
For this and all other statements see the survey \cite{Pr08}. Here $J_n(f_n)$ is the $n$th {\it discrete multiple integral} of $f_n$ and is defined by
$$
J_n(f_n) := \sum_{(i_1, \ldots, i_n) \in \N^n} f_n(i_1, \ldots, i_n) Y_{i_1} \cdots Y_{i_n} \, 1_{\Delta_n}(i_1, \ldots, i_n).
$$

Writing $F = f(X) = f(X_1,X_2,...)$ we define the {\it discrete gradient} $D_kF$ of $F$ at $k$th coordinate:
\begin{align*}
D_kF & := \sqrt{p_k q_k} (F_k^+ - F_k^-), \\
DF & := (D_1F, D_2F,...),
\end{align*}
where $F_k^+ := f(X_1,...,X_{k-1},+1,X_{k+1},...)$ and $F_k^- := f(X_1,...,X_{k-1},-1,X_{k+1},...), \, k \in \N$.
We will consider
\begin{align*}
F \in \D^{1,2} := \{F \in L^2(\Omega, \sigma(X), \Pp) \vert \, \mathbb{E}[\norm{DF}_{l^2(\N)}^2] < \infty \},
\end{align*}
where $\mathbb{E}[\norm{DF}_{l^2(\N)}^2] := \mathbb{E}\bigl[ \sum_{k \in \N} (D_k F)^2 \bigr]$.
The {\it adjoint operator} $\delta$ of $D$ is characterized by the duality relation
\begin{align*}
	\mathbb{E}[\langle DF, u \rangle] = \mathbb{E}[F \delta(u)], \quad F \in \D^{1,2}, u \in \text{Dom}(\delta).
\end{align*} 
Here the domain $\text{Dom}(\delta)$ consists of square-integrable $\mathcal H$-valued random variables $u \in L^2(\Omega, \mathcal H)$ satisfying
the property, that there is some constant $C_u >0$ such that $| \E \langle DF, u \rangle | \leq C_u \sqrt{\E(F^2)}$ for all $F \in \D^{1,2}$.

Supposing $F$ has the representation \eqref{chaos3} we say $F \in \text{Dom}(L)$ if $\sum_{n \in \N} n^2 n! \|f_n\|^2 < \infty$. In this case we define  
\begin{align*}
LF  := \sum_{n=1}^\infty -n J_n(f_n),
\,\, L^{-1}F := \sum_{n=1}^\infty - \frac{1}{n}J_n(f_n),
\end{align*}
the {\it Ornstein--Uhlenbeck operator} $L$ and the {\it pseudo-inverse Ornstein--Uhlenbeck operator} $L^{-1}$. You can show that $F \in Dom(L)$ is equivalent to $F \in \D^{1,2}$ and $DF \in Dom(\delta)$, in this case, it is $L = -\delta D$.

A first normal approximation bound for symmetric Rademacher functionals has been obtained in \cite{NPR10}. A bound in the Kolmogorov distance has later been found
in \cite{KRT16} and was extended to the non-symmetric setting in \cite{KRT17} and \cite{DK19}. In \cite{ERTZ22}, 
a significantly simplified bound was presented:

\begin{theorem}[Berry--Esseen bound for Rademacher functionals, \cite{ERTZ22}]\label{T3}
	Let $F \in \D^{1,2}$ with $\mathbb{E}[F] = 0, \Var(F) = 1$ and 
	\begin{align*}
		F f_z(F) + \Unens{F>z} \in \D^{1,2} \quad \forall z \in \R.
	\end{align*}
	Assume in addition that 
	\begin{align*}
		\frac{1}{\sqrt{pq}} DF \abs{D L^{-1} F}  \in Dom(\delta).
	\end{align*}
	Then one has 
	\begin{align*}
\sup_{z \in \R} \abs{\Pp(F \leq z) - \Phi(z)} \leq \mathbb{E} \big[ \big| 1 - \langle DF , - DL^{-1} F \rangle \big| \big] + 
2 \|  \delta \bigl( \frac{1}{\sqrt{pq}} DF \abs{D L^{-1} F} \bigr) \|_{L^1(\Omega)}.
	\end{align*}
\end{theorem}
\bigskip

\noindent
Our result is:
\begin{theorem}[Non-uniform Berry--Esseen bound for Rademacher functionals]\label{Theorem:NU_Rademacher}
	Let $F \in \D^{1,2}$ with $\mathbb{E}[F] = 0, \Var(F) = 1$ and $\mathbb{E}[F^{2k}] < C$ for fixed $k \in \N$. Further
	\begin{align*}
		F f_z(F) + \Unens{F>z} \in \D^{1,2} \quad \forall z \in \R,
	\end{align*}
	\begin{align*}
		\frac{1}{\sqrt{pq}} DF \abs{D L^{-1} F}  \in Dom(\delta).
	\end{align*}
	Then, for any $z \in \R$,
	\begin{align*}
\abs{\Pp(F \leq z) - \Phi(z)} \leq \frac{C}{(1 + \abs{z})^k} \Bigg (& \prth{\mathbb{E} \prth{ 1 - \langle DF , - DL^{-1} F \rangle }^2 }^{1/2} \\
& + \prth{\mathbb{E} \prth{  \delta\prth{ \frac{1}{\sqrt{pq}} DF \abs{D L^{-1} F} } }^2 }^{1/2} \Bigg ),
	\end{align*}
and $C$ is a constant depending on $k \in \N$.
\end{theorem}
\bigskip

\section{Proofs of non-uniform Berry-Esseen bounds}

We prove the three non-uniform Berry-Esseen bounds. The proof will focus on bounding $\E | f'_z(F)|$, where $f_z$ is the solution of the Stein-equation
\eqref{Steinequation}. One of the main steps is adapting and refining Lemma 5.1 in \cite{CS01}. We will prepare a little technical detail before.

\begin{lemma} \label{lemma}
Let $k \in \N$ and $z \in \R$ be such that $| z^l | \leq C$ for any $1 \leq l \leq 2k$. For any constant $C_1 >0$ there exists a constant $C_2 >0$ such that
\begin{align}\label{minimum}
	\min\prth{1, \frac{C_1}{\abs{z}^{2k}}} & \leq \frac{C_2}{(1+\abs{z})^{2k}}.
\end{align}

\end{lemma}

\begin{proof}
If $\abs{z}^{2k} \leq C_1$, the minimum is 1 and \eqref{minimum} is equivalent to $(1 + \abs{z})^{2k} \leq C_2$,
which is true by our assumption. 
If $\abs{z}^{2k} > C_1$, the minimum is $\frac{C_1}{\abs{z}^{2k}}$ and \eqref{minimum} is equivalent to
\begin{align*}
	C_1 \frac{(1+\abs{z})^{2k}}{\abs{z}^{2k}} \leq C_2,
\end{align*}
and
\begin{align*}
	C_1 \frac{(1+\abs{z})^{2k}}{\abs{z}^{2k}} &  = C_1 \sum_{l=0}^{2k} \binom{2k}{l}\abs{z}^{l-2k} \\
	                                                                                         &  \leq C_1 \prth{\frac{1}{\abs{z}^{2k}}+ \cdots + C_{2k,s}\frac{1}{\abs{z}^{s}}+ \cdots + 1} \\
	                                                                                         &  < C_1 \prth{\frac{1}{C_1} + \cdots + C_{2k,s}\frac{1}{C_1^{s/2k}}+ \cdots + 1} \\
	                                                                                         & =  1 + \cdots + C_{2k,s} C_1^{1 - s/2k}+ \cdots + C_1 =: C_2.
	                                                                                        \end{align*}
Here $C_{2k,s}$ is a constant to bound the corresponding binomial coefficient.
\end{proof}

\begin{proof}[Proof of Theorem \ref{Gaussianmain}]
In the following all objects and notions we are dealing with are presented in subsection \ref{subGauss}.
Let us recall the major step in the proof of \cite[Proposition 5.1.1]{NP09}. In the context of Theorem \ref{Gaussianmain} we have for any $z \in \R$ that
$$
\Pp(F \leq z) - \Phi(z) = \E \bigl( f_z(F) - F f_z(F) \bigr) = \E \bigl( f'_z(F) ( 1 - \langle DF, - DL^{-1} F \rangle_{\mathcal H} ) \bigr).
$$
Remember that $f'_z(w) = (1-\Phi(z)) \bigl(1 + \sqrt{2 \pi} w e^{w^2/2}  \Phi(w) \bigr)$ for $w \leq z$.
As in \cite[Lemma 5.1]{CS01} we consider the decomposition
$$
f'_z(F) =  f'_z(F) 1_{\{F < 0\}} + f'_z(F) 1_{ \{0 \leq F \leq z/2\}} + f'_z(F) 1_{ \{F > z/2 \}}.
$$
Let $z >0$. With \eqref{first} we obtain $f'_z(F) 1_{\{F < 0\}} \leq 1 - \Phi(z)$ and by Mill's ratio we have $1 - \Phi(z) \leq \min \bigl( \frac 12, \frac 1z \bigr) e^{-z^2/2}$,
hence $f'_z(F) 1_{\{F < 0\}} \leq \frac 12 e^{-z^2/2}$. Moreover with \eqref{first}, on the set $\{0 \leq F \leq z/2\}$ we have $f'_z(F) \leq C  e^{-z^2/4}$ with some
constant $C >0$. Summarizing we obtain for every $z>0$ that
$$
f'_z(F) \leq C e^{-z^2/4} +  f'_z(F) 1_{ \{F > z/2 \}}.
$$
Now it follows with $\| f'_z \|_{\infty} \leq 1$ that
$$
| \Pp(F \leq z) - \Phi(z) | \leq C e^{-z^2/4}  \E | 1 - \langle DF, - DL^{-1} F \rangle_{\mathcal H} | + \E  \big( |1 - \langle DF, - DL^{-1} F \rangle_{\mathcal H} | 1_{ \{F > z/2 \}} \bigr).
$$
Under the $2k$-moment condition we have $\Pp (F > z/2) \leq \frac{\E (F^{2k})}{(z/2)^{2k}} \leq \frac{C}{z^{2k}}$.
Now we use the fact that  $|f'_z(x)| \leq 1$ for any $x \in \R$. Therefore  
it is possible to consider the minimum of our bounds and 1 for all substantial subterms. Applying Lemma \ref{lemma} finishes the proof of our Theorem for $z>0$ by applying the Cauchy-Schwarz inequality. 

Next let us consider the case $z \leq 0$. Here we obtain that $\Pp(F \leq z) - \Phi(z) = \Pp(-F \leq -z) - \Phi(-z)$ by the symmetry
of the standard normal distribution. But now we can apply our bound for $-z >0$, using $\langle D(-F), -D L^{-1}(-F) \rangle_{\mathcal H} =
\langle DF, -D L^{-1} F \rangle_{\mathcal H}$. Hence the Theorem is proved.
\end{proof}

\begin{proof}[Proof of Theorem \ref{Poissonmain}]
In the following all objects and notions we are dealing with are presented in subsection \ref{subPoisson}.
By Stein's method and the proof of \cite[Theorem~1.12 ]{LRPY20} (special case $\mathbb{E}[F]=0, \sigma^2=1$) we have for $z \in \R$
\begin{align}\label{result_ERTZ}
\abs{ \Pp(F \leq z) - \Phi(z)} & = \abs{\mathbb{E}[f_z^\prime(F) - Ff_z(F)]} \leq J_1 + J_2 
\end{align}
with
\begin{align*}
J_1 & := \mathbb{E} \abs{f_z^\prime(F) \prth{ 1 - \langle DF , - DL^{-1} F \rangle }}, \\
J_2 & := \mathbb{E}\left[(F f_z(F) + \Unens{F>z} )\delta\prth{ DF \abs{D L^{-1} F} }\right].
\end{align*}
The first term $J_1$ can be handled exactly as in the proof of Theorem \ref{Gaussianmain}. But we will change
the arguments a little. For the upcoming estimation we can split $J_2$ into two terms, namely
\begin{align*}
\abs{J_2} & \leq \mathbb{E} \left[ \abs{  \delta\prth{ DF \abs{D L^{-1} F} } } \abs{ F f_z(F) + \Unens{F>z}}\right] \leq  J_{21} + J_{22}
\end{align*}
with
\begin{align*}
		J_{21} & := \mathbb{E} \left[ \abs{  \delta\prth{ DF \abs{D L^{-1} F} }  } \abs{ F f_z(F)}\right], \\
		J_{22} & := \mathbb{E} \left[ \abs{  \delta\prth{ DF \abs{D L^{-1} F} }  } \Unens{F>z}\right].
\end{align*}
	
As a first step we apply Cauchy--Schwarz inequality on every of our terms of interest
 \begin{align*}
 J_1 & \leq (\mathbb{E} \abs{f_z^\prime(F)}^2)^{1/2} \prth{\mathbb{E} \prth{ 1 - \langle DF , - DL^{-1} F \rangle }^2 }^{1/2}, \\
   J_{21} & \leq (\mathbb{E} \abs{ F f_z(F)}^2)^{1/2} \prth{\mathbb{E} \prth{  \delta\prth{ DF \abs{D L^{-1} F} } }^2 }^{1/2}, \\
    	J_{22} & \leq (\mathbb{P}(F>z))^{1/2} \prth{\mathbb{E} \prth{  \delta\prth{ DF \abs{D L^{-1} F} } }^2 }^{1/2}.
 \end{align*}
To bound the factors where the solutions of the Stein-equation $f_z$ appear, we refer to the proof of \cite[Theorem~2]{LLWC21}, where the authors considered even
some non-normal limiting distributions. Their computations are valid for the normal case by choosing $g(x) = x$ in their framework, see \cite[Section 2]{LLWC21}. In particular, in their condition (A4) one can choose $\tau$ arbitrary, e.g. $\tau = \frac{1}{2}$. We extend our proof adaption by also referring to the proof of \cite[Theorem~2]{LLWC21}. Note, that our appearing constants will very likely depend on $k \in \N$, but we will just write $C$. The core of our proof is to show
\begin{align}
		(\mathbb{E} \abs{f_z^\prime(F)}^2)^{1/2} & \leq \frac{C}{(1 + \abs{z})^k}, \label{Stein_fragment_1}\\
		(\mathbb{E} \abs{ F f_z(F)}^2)^{1/2} & \leq \frac{C}{(1 + \abs{z})^k}, \label{Stein_fragment_2} \\
		 (\mathbb{P}(F>z))^{1/2} & \leq \frac{C}{(1 + \abs{z})^k}. \label{Stein_fragment_3}
\end{align}
\medskip

\noindent
To prove these bounds we first assume that $z > 0$:
\\ Proof of \eqref {Stein_fragment_1}:  As in the proof of the bound (12) in \cite{LLWC21} and similar to the
proof of Theorem \ref{Gaussianmain}, we adapt Lemma 5.1 from \cite{CS01} and consider the following decomposition:
\begin{align*}
	\mathbb{E} \abs{f_z^\prime(F)}^2 & = \mathbb{E} [f_z^\prime(F)^2 \Unens{F \leq 0 }] + \mathbb{E} [f_z^\prime(F)^2 \Unens{0 < F \leq z/2 }] + \mathbb{E} [f_z^\prime(F)^2 \Unens{F >  z/2 }].
\end{align*}
and bound these terms. First, as in the proof of Theorem \ref{Gaussianmain}, we obtain
	\begin{align*}
	\mathbb{E} [f_z^\prime(F)^2 \Unens{F \leq 0 }] & \leq (1-\Phi(z))^2 \leq \frac 14 e^{-z^2} \leq \frac{C}{z^{2k}},
\end{align*}
using in particular $z^{2l} e^{-z^2} \leq C \; \forall \; l \in \N$. Second, again as in the proof of Theorem \ref{Gaussianmain}, we obtain
	\begin{align*}
		\mathbb{E} [f_z^\prime(F)^2 \Unens{0 < F \leq z/2 }] & \leq  C e^{-z^2/2} \leq \frac{C}{z^{2k}}.
	\end{align*}
Last, we can bound the third term and show also \eqref {Stein_fragment_3} by using $\norm{f_z^\prime}_{\infty} \leq 1$ and the Markov inequality for
	\begin{align*}
	\mathbb{P}(F>z) & \leq \mathbb{P}(F^{2k} > z^{2k}) \leq \frac{\mathbb{E}[F^{2k}]}{z^{2k}} < \frac{C}{z^{2k}} \quad \text{for all} \,\,  \; z > 0.
           \end{align*}
Next we prove \eqref {Stein_fragment_2}:
We adapt the proof of (17) in \cite{LLWC21}. We us Stein's identity $x \, f_z(x) = f'_z(x) - (1_{\{ x \leq z\}} - \Phi(z))$ and our estimate on $
\mathbb{E} \left[\prth{f_z^\prime(F)}^2\right]$ to observe
	\begin{align*}
		\mathbb{E} [(Ff_z(F))^2] & = \mathbb{E} \left[\prth{f_z^\prime(F) -  \prth{\Unens{F \leq z }- \Phi(z)}}^2\right] \\
		& \leq C \prth{\mathbb{E} \left[\prth{f_z^\prime(F)}^2\right]  + \mathbb{E} \left[ \prth{\Unens{F \leq z }- \Phi(z)}^2\right] } \\
		& \leq C \prth{\frac{C}{z^{2k}}  + \mathbb{E} \left[ \prth{\Unens{F \leq z }- \Phi(z)}^2\right] }.
	\end{align*}
	For $z>0$ we have
	\begin{align*}
		\mathbb{E} \left[ \prth{\Unens{F \leq z }- \Phi(z)}^2\right] & = \mathbb{E} \left[ \prth{1- \Phi(z)}^2 \Unens{F \leq z } \right] + \mathbb{E} \left[\Phi(z)^2 \Unens{F > z }\right] \\ & \leq  \prth{1- \Phi(z)}^2 + \Pp(F >z) \leq \frac{C}{z^{2k}}.
	\end{align*}
Now we use the fact that $\| x f_z(x) \| \leq 1$ and $\|f'_z(x)\| \leq 1$ for any $x \in \R$. Therefore  
it is possible to consider the minimum of our bounds and 1 for all substantial subterms. Applying Lemma \ref{lemma} finishes the proof of our Theorem for $z>0$. 
\medskip

\noindent
Next we treat the case $z \leq 0$. With the simple identity $1 = \Unens{F>z} + \Unens{F \leq z}$, we receive the alternativ bound on $| \Pp(F \leq z) -\Phi(z)|$
by $J_1$ as before and $J_2$ changes to 
$$\mathbb{E}\left[(F f_z(F) + \Unens{F \leq z} )\delta\prth{ DF \abs{D L^{-1} F} }\right].$$
A direct consequence is the need to modify \eqref {Stein_fragment_3}. We obtain 
\begin{align*}
\mathbb{P}(F\leq z) & \leq \mathbb{P}(F^{2k} > z^{2k}) \leq \frac{\mathbb{E}[F^{2k}]}{z^{2k}} < \frac{C}{z^{2k}} \quad \forall \; z \leq 0.
\end{align*}
Next we will prove \eqref {Stein_fragment_1} and \eqref {Stein_fragment_2} for $z \leq 0$. As the authors explain themselves at the end of the proof of \cite[Theorem~2]{LLWC21} no big modifications of their argumentation are needed. For $z \leq 0$ we know have that
\begin{align*}
	\mathbb{E} \abs{f_z^\prime(F)}^2 & = \mathbb{E} [f_z^\prime(F)^2 \Unens{F \leq z/2 }] + \mathbb{E} [f_z^\prime(F)^2 \Unens{z/2 < F \leq 0 }] + \mathbb{E} [f_z^\prime(F)^2 \Unens{F >  0 }].
\end{align*}
Remember that for $x > z$ we obtain $f'_z(x) = \bigl( \sqrt{2 \pi} x e^{x^2/2}(1- \Phi(x)) -1 \bigr) \Phi(z)$. This time we
will apply the bound 
$$
\Phi(z) \leq \frac{e^{-z^2/2}}{\sqrt{2 \pi} |z|}, \,\, z \leq 0,
$$
to check \eqref{Stein_fragment_1}.  To prove \eqref{Stein_fragment_2}, we will estimate
$$
\E \bigl( 1_{\{ F \leq z \}} - \Phi(z) \bigr)^2 \leq 2 \Pp (F \leq z) + 2 \Phi^2(z) \leq \frac{C}{|z|^2}.
$$
Now the Theorem is proved. 
\end{proof}

\begin{proof}[Proof of Theorem~\ref{Theorem:NU_Rademacher}]
In the following all objects and notions we are dealing with are presented in subsection \ref{subRademacher}.
By Stein's method and the proof of \cite[Theorem~3.1]{ERTZ22} we have for $z \in \R$
	\begin{align}\label{result_ERTZ}
		\abs{ \Pp(F \leq z) - \Phi(z)} & = \abs{\mathbb{E}[f_z^\prime(F) - Ff_z(F)]} \leq J_1 + J_2 
	\end{align}
	with
	\begin{align*}
		J_1 & := \mathbb{E} \abs{f_z^\prime(F) \prth{ 1 - \langle DF , - DL^{-1} F \rangle }}, \\
		J_2 & := \mathbb{E}\left[(F f_z(F) + \Unens{F>z} )\delta\prth{ \frac{1}{\sqrt{pq}} DF \abs{D L^{-1} F} }\right].
	\end{align*}
Since we consider normal approximation, from here on the proof is just an adaption of the proof of Theorem~\ref{Poissonmain}.
\end{proof}
\bigskip

\section{Non-uniform bounds of Poincar\'e-type}

The bounds in Theorem \ref{T1}, \ref{T2} and \ref{T3} can be quite difficult to evaluate if on uses the representations of the inverse Ornstein-Uhlenbeck generators
in terms of the chaos expansions, because this requires the explicit knowledge of the kernel functions in these representations. The main tool
to overcome this problems is the development of Mehler's formula in combination with the second order Poincar\'e inequalities. Then the estimates
of the quality of normal approximation only involve the first and second Malliavin derivative and the first and second order difference operators, respectively. 
Our non-uniform bounds can easily be transmitted
to the second order Poincar\'e inequalities.

For an isonormal Gaussian process $X$ as in Section \ref{subGauss}, we assume $F \in \D^{2,4}$. Then the Kolmogorov distance between
$F$ and a standard normal $Z$ can be bounded by a constant times
$$
\E \big[ \|D^2 F \|_{op}^4 \bigr]^{1/4} \times \E \big[ \| DF \|_{\mathcal H}^4 \big]^{1/4},
$$
where $ \|D^2 F \|_{op}$ is the operator norm of the random Hilbert-Schmidt operator $f \mapsto \langle f, D^2F \rangle_{\mathcal H}$.
With  Theorem \ref{Gaussianmain}. we obtain a non-uniform Berry-Esseen bound with prefactor $C / (1+|x|^3)$.

In the Poisson setting of Section \ref{subPoisson}, \cite[Theorem 1.2]{LPS16} can be simplified applying \cite[Theorem 1.15]{LRPY20}.  
We introduce the notion
$$
D_{x,y}^2 F := D_y(D_x F).
$$
A direct consequence of Theorem~\ref{Poissonmain} is:

\begin{theorem}[Non-uniform Berry--Esseen bound for Poisson functionals, second version]\label{Poissonmain2}
Let $F \in {\hat{\D}}^{1,2}$ with $\mathbb{E}[F] = 0, \Var(F) = 1$ and $\mathbb{E}[F^{2k}] < \infty$ for fixed $k \in \N$. 
Further
\begin{align*}
	\mathbb{E} \int_X \int_X \left[D_y \left(D_xF \abs{D_x L^{-1} F}\right)\right]^2 \mu^2(dx,dy) < \infty,
\end{align*}
\begin{align*}
F f_z(F) \in {\hat{\D}}^{1,2} \quad \forall z \in \R.
\end{align*}
Then, for any $z \in \R$,
	\begin{align*}
\abs{\Pp(F \leq z) - \Phi(z)} \leq \frac{C}{(1 + \abs{z})^k} \biggl( 2 \sqrt{A_1} + \sqrt{A_2} + \sqrt{A_3} + \sqrt{A_4  + A_5} \biggr),
	\end{align*}
and $C$ is a constant depending on $k \in \N$, where
\begin{align*}
A_1 & := \int \big[ \E (D_{x_1} F)^2 (D_{x_2} F)^2 \bigr]^{1/2}  \big[ \E (D_{x_1,x_3}^2 F)^2 (D_{x_2,x_3}^2 F)^2 \bigr]^{1/2} \mu^3(d(x_1,x_2,x_3)),\\
A_2 & := \int  \E(D_{x_1,x_3}^2 F)^2 (D_{x_2,x_3}^2 F)^2 \mu^3(d(x_1,x_2,x_3)), \\
A_3 & := \int \E(D_x F)^4 \mu(dx), \\
A_4 & :=  6 \int \big[ \E (D_{x_1}F)^4 \bigr]^{1/2}  \big[ \E (D_{x_1,x_2}^2F)^4 \bigr]^{1/2}  \mu^2(d(x_1,x_2)), \\
A_5 & := 3 \int \E (D_{x_1,x_2}^2F)^4  \mu^2 (d(x_1,x_2)).
\end{align*}
\end{theorem}

\begin{proof}
In the proof of \cite[Theorem~1.2]{LPS16} it is shown that
\begin{align*}
\prth{\mathbb{E} \prth{ 1 - \langle DF , - DL^{-1} F \rangle }^2 }^{1/2} & \leq  2 \sqrt{A_1} + \sqrt{A_2}, \\
\prth{\mathbb{E} \prth{  \delta \prth{ DF \abs{D L^{-1} F} }}^2 }^{1/2} & \leq \sqrt{A_3} + \sqrt{A_4  + A_5}.
\end{align*}
Now the Theorem follows with Theorem \ref{Poissonmain}.
\end{proof}

\begin{remark}
Compared to the bound in  \cite[Theorem1.2]{LPS16}, we remark that the third and the fourth term
$$
\gamma_3 := \int \E|D_xF|^3 \mu(dx),
$$
$$
\gamma_4 := \frac 12 ( \E F^4)^{1/4} \int \bigl[ \E (D_x F)^4 \bigr]^{3/4} \mu(dx),
$$
there do not appear in our result. In fact the two terms were dominating in some results. Removing these terms thus lead to an improvement.
\end{remark}

\noindent
The same is true in the Rademacher setting of Section \ref{subRademacher}.

\begin{theorem}[Non-uniform Berry-Esseen bound for Rademacher functionals, second version] \label{Rademachermain2}
	Let $F \in \D^{1,2}$ with $\mathbb{E}[F] = 0, \Var(F) = 1$ and $\mathbb{E}[F^{2k}] < C$ for fixed $k \in \N$. Further
	\begin{align*}
		F f_z(F) + \Unens{F>z} \in \D^{1,2} \quad \forall z \in \R,
	\end{align*}
	\begin{align*}
		\frac{1}{\sqrt{pq}} DF \abs{D L^{-1} F}  \in Dom(\delta).
	\end{align*}
	Then, for any $z \in \R$,
	\begin{align*}
		\abs{\Pp(F \leq z) - \Phi(z)} \leq \frac{C}{(1 + \abs{z})^k} \prth{\frac{\sqrt{15}}{2} \sqrt{B_1} + \frac{\sqrt{3}}{2} \sqrt{B_2} + 2 \sqrt{B_3} + 2 \sqrt{6} \sqrt{B_4} + 2 \sqrt{3} \sqrt{B_5} },
	\end{align*}
where
\begin{align*}
B_1 & := \sum_{j,k,l \in \N} \sqrt{\mathbb{E}[(D_j F)^2 (D_k F)^2]} \sqrt{\mathbb{E}[(D_l D_j F)^2 (D_l D_k F)^2]}, \\
B_2 & := \sum_{j,k,l \in \N} \frac{1}{p_l q_l} \mathbb{E}[(D_l D_j F)^2 (D_l D_k F)^2], \\
B_3 & := \sum_{k \in \N} \frac{1}{p_k q_k} \mathbb{E}[(D_k F)^4], \\
B_4 & := \sum_{k,l \in \N} \frac{1}{p_k q_k} \sqrt{\mathbb{E}[(D_k F)^4]} \sqrt{\mathbb{E}[(D_l D_k F)^4]}, \\
B_5 & := \sum_{k,l \in \N} \frac{1}{p_k q_k p_l q_l} \mathbb{E}[(D_l D_k F)^4],
\end{align*}
and $C$ is a constant depending on $k \in \N$.
\end{theorem}

\begin{proof}[Proof of Theorem~\ref{Rademachermain2}]
In the proof of \cite[Theorem~4.1]{ERTZ22} it is shown that
\begin{align*}
\prth{\mathbb{E} \prth{ 1 - \langle DF , - DL^{-1} F \rangle }^2 }^{1/2} & \leq \frac{\sqrt{15}}{2} \sqrt{B_1} + \frac{\sqrt{3}}{2} \sqrt{B_2}, \\
\prth{\mathbb{E} \prth{  \delta\prth{ \frac{1}{\sqrt{pq}} DF \abs{D L^{-1} F} } }^2 }^{1/2} & \leq 2 \sqrt{B_3} + 2 \sqrt{6} \sqrt{B_4} + 2 \sqrt{3} \sqrt{B_5}.
\end{align*}
Now the Theorem follows with Theorem \ref{Theorem:NU_Rademacher}.
\end{proof}
\bigskip

\section{Applications} \label{Examples}

\subsection{Fractional Brownian motion}
In our first application we consider the situation of an isonormal Gaussian process as in subsection \ref{subGauss}.
A fractional Brownian motion $B^H=(B_t^H : t \geq 0)$ with Hurst index $0 < H <1$ is a continuous-time centered Gaussian process with covariance
$$
\E [ B_t^H, B_s^H ] = \frac 12 (t^{2H} + s^{2H} -|t-s|^{2H}), \quad s,t \geq 0,
$$
see \cite{Nourdin} for details and background material. 
It is known (see \cite[Theorem 2.1]{Nourdin}) that
$$
S_n := \sum_{k=0}^{n-1} \bigl( B^H_{\frac{k+1}{n}} - B^H_{\frac kn} \bigr)^2, \quad n \geq 1,
$$
is such that $n^{2H-1} S_n$ converges in probability to 1, so that
$$
\hat{H}_n = \frac 12 - \frac{\log S_n}{2 \log n}, \quad n \geq 2,
$$
is a reasonable estimator for the Hurst index. We now define
$$
F_n := \frac{n^{2H}}{\sigma_n} \sum_{k=0}^{n-1} \big[ \bigl( B^H_{\frac{k+1}{n}} - B^H_{\frac kn} \bigr)^2 - n^{-2H} \bigr], \quad n \geq 1,
$$
where $\sigma_n>0$ is chosen such that $\E(F_n^2)=1$. Again the behaviour of $F_n$ controls the error of the estimator $\hat{H}_n$. Notice that
by the selfsimilarity of the fractional Brownian motion $F_n$ has the same law as 
$$
F_n = \frac{1}{\sigma_n} \sum_{k=0}^{n-1} \big[ \bigl( B^H_{k+1} - B^H_{k} \bigr)^2 - 1 \bigr].
$$
Now there exists a stochastic representation of the fractional Brownian motion as a Wiener integral of a two-sided Brownian motion $W$, 
meaning that $\{ B_{k+1}^H - B_k^H, k \in \N \}$ is equal in law to $\{ \int_{\R} e_k(s) dW_s, k \in \N \}
= \{ I_1^W(e_k), k \in \N \}$ for certain $(e_k)_k$, see \cite[Proposition 2.3]{Nourdin}. It follows the representation $F_n = I_2^W(f_n)$ with $f_n = \frac{1}{\sigma_n} \sum_{k=0}^{n-1} e_k \otimes e_k$, where $e_k \otimes e_k$ denotes the tensor product. Now asymptotic normality for $F_n$ together with rates of convergence for the Kolmogorov distance has been investigated, see \cite[Theorem 6.3]{Nourdin}. The result in \cite{Nourdin} was given in terms of the total variation distance,
but with the help of \cite[Theorem 5.1.3]{NP12} the result is the same for both distances, beside a factor 2. In fact, while for $H > 3/4$, the sequence $(F_n)_n$ does not satisfy a central limit theorem, for $0<H\leq 3/4$ is holds that
$$
\sup_{z \in \R} | \Pp(F_n \leq z) - \Phi(z)|  \leq A_n, \quad n \geq 2,
$$
and $A_n$ is given by
\begin{equation}
A_n = c_H \times \, \left\{ \begin{array}{ll} \frac{1}{\sqrt{n}} & \mbox{:} \,\, 0 < H < \frac 58 \\ \frac{(\log n)^{3/2}}{\sqrt{n}} & \mbox{:} \,\, H = \frac 58 \\
\frac{1}{n^{3-4H}} & \mbox{:} \,\, \frac 58 < H < \frac 34 \\
\frac{1}{\log n} & \mbox{:} \,\, H = \frac 34 
\end{array} \right.
\end{equation}
with a constant $c_H$ only depending on $H$. As a consequence, one can show that for $0 < H \leq \frac 34$ both of the rescaled 
random variables $\sqrt{n} (n^{2H-1} S_n -1)$ and $\sqrt{n} \log n (\hat{H}_n -H)$ are, as $n \to \infty$, normally distributed with explicitly known limiting variances, see \cite[Section 6.4]{Nourdin}.

With \cite[Theorem 5]{STx} we know that $c_n=c_n(q,f_n)$ in \eqref{conGauss}, given a Wiener-It\^o integral $I_q(f_n)$, is given by
$$
c_n(q,f_n) = \biggl( q^{3q/2} \bigl( \max_{r=1, \ldots, q-1} \| f_n \otimes_r f_n\|_{\mathcal H^{\otimes 2(q-r)}} \bigr)^{\alpha(q)} \biggr)^{-1}.
$$
Here $\otimes_r$ stands for the $r$th contraction operator, see \cite[Definition 5.4]{Nourdin}. Moreover $\alpha(q) = \frac{q+2}{3q+2}$, if $q$ is even, and
$\alpha(q) = \frac{q^2-q-1}{q(3q -5)}$, if $q$ is odd. In our situation we have $q=2$, hence for $f_n = \frac{1}{\sigma_n} \sum_{k=0}^{n-1} e_k \otimes e_k$ we obtain
$$
c_n(2,f_n) = \bigl( 8 \|f_n \otimes f_n \|_{\mathcal H^{\otimes 2}}^{1/2} \bigl)^{-1}.
$$
From the proof of \cite[Theorem 6.3]{Nourdin} we obtain that  $\|f_n \otimes f_n \|_{\mathcal H^{\otimes 2}} \leq \frac{1}{\sqrt{8}} A_n$.
Hence $c_n(2,f_n) \leq 2^{-9/4} A_n^{-1/2}$ and with \cite[Corollary 2]{STx} we obtain
$$
 \Pp \bigl( | F_n | \geq z \bigr) \leq 2 \exp \biggl( - \frac 14 \min \bigg\{ \frac{z^2}{2}, (2^{-9/4} A_n^{-1/2}  z ) \bigg\} \biggr).
$$
Hence we have proved:

\begin{theorem}[Non-uniform Berry-Esseen bound for a discretized version of the quadratic variation of a fractional BM]
	In the setting above we have
	\begin{align*}
		\abs{\Pp(F_n \leq z) - \Phi(z)} \leq \biggl( \sqrt{2} \exp \biggl( - \frac 18 \min \bigg\{ \frac{z^2}{8}, (2^{-13/4} A_n^{-1/2}  z ) \bigg\} \biggr) + c \exp(-z^2/4) \biggr) A_n.
\end{align*}
\end{theorem}
\medskip

\begin{remark}
In \cite{STx}, two other examples are presented, where the sequence $c_n(q,f_n)$ was calculated explicitly. This is Theorem 7, were
explosive integrals of a Brownian sheet are represented as a Wiener-It\^o integral of degree 2 with respect to the Brownian sheet. Another example
is presented in \cite[Theorem 10]{STx}, where the sample bispectrum of a spherical Gaussian random field  is represented as an element of the third Wiener chaos.
For both examples, we would be able to present the corresponding non-uniform Berry-Esseen bound. We skip the details.
\end{remark}
\medskip

\subsection{Subgraph counts in random geometric graphs}
In our second application we consider the situation of a Poisson process as in subsection \ref{subPoisson}. Let $(t_n)_n$ be a positive real-valued sequence such that $t_n \to \infty$, as $n \to \infty$, and let $(r_n)_n$ be a positive real-valued sequence. Fix a convex, compact set $W \subset \R^d$ with interior points, and let, for each $n \in \N$, $\eta_n$ be a Poisson process whose intensity measure $\mu_n$ is $t_n$ times the restriction of the Lebesgue measure to $W$. The random geometric
graph $RGG(\eta_n, r_n)$ is constructed by taking the points of $\eta_n$ as the vertices of the graph and by connecting two distinct points by an edge whenever their Euclidean distance is strictly positive and does not exceed the given threshold $r_n$. We are interested to count subgraphs. Let $G$ be a fixed connected graph
with $q$ vertices, and for $x_1, \ldots, x_q \in W$ we let $f(x_1, \ldots, x_q, G, r_n)$ be $1/q!$ times the number of subgraphs of $RGG(\{x_1, \ldots, x_q\},r_n)$
that are isomorphic to $G$. Now the subgraph counting statistic is given by
$$
S_n(G) := \sum_{(x_1, \ldots, x_q) \in \eta_{n, \not=}^q} f(x_1, \ldots, x_q, G, r_n),
$$
where $ \eta_{n, \not=}^q$ stands for the set of all $q$-tuples of distinct points of $\eta_n$. $S_n(G)$ is the number of non-induced copies of $G$ in $RGG(\eta_n,r_n)$.
Now a concentration inequality of $(S_n(G))_n$ can be discussed. We will have the following choice of $c(q,f,\alpha_n)$ in \eqref{conPoisson}. 
In the model it can be proved that there
exists a constant $v>0$ such that
$$
\Var(S_n(G)) \geq v \, \max \{ t_n^{q-1} (\kappa_d r_n^d)^{2q-2}, t_n^q (\kappa_d r_n^d)^{q-1} \, \}, \quad n \in \N.
$$
Here $\kappa_d$ denotes the volume of the Euclidean unit ball in $\R^d$.
Now the sequence $c(q,f,\alpha_n)$ is given in terms of $v, t_n, r_n, q$ and $\Vol(W)$ by
$$
c(q,f,\alpha_n) = \frac{\sqrt{v t_n \min\{1, t_n \kappa_d r_n^d\}^{q-1}}}{q^{3q} \max\{1, \Vol(W)/v \, \} }, \quad n \in \N,
$$
see Corollary 4.2 and Theorem 3.2 in \cite{ST23}.
With \eqref{conPoisson} we will be able to bound 
$$ 
\sqrt{ \Pp (S_n(G) > z/2)} + c e^{-x^2/4},
$$ the factor as a function in $z$ in front of the
Berry-Essen bound. We have proved:

\begin{theorem}[Non-uniform Berry-Esseen bound for subgraph counts in random geometric graphs]
	In the setting above we have
	\begin{eqnarray*}
		\abs{\Pp(S_n(G) \leq z) - \Phi(z)} & \leq & \biggl( \sqrt{2} \exp \biggl( - \frac 18 \min \bigg\{ \frac{z^2}{2^{q+2}}, \frac{z^{1/q}}{2^{1/q}} \biggl(
		\frac{\sqrt{v t_n \min\{1, t_n \kappa_d r_n^d\}^{q-1}}}{q^{3q} \max\{1, \Vol(W)/v \, \} }\biggr)^{1/q}  \bigg\} \biggr) \\
		 &  &  \hspace{2cm} + c e^{-z^2/4} \biggr) A_n,
\end{eqnarray*}
where $A_n$ is the known Berry-Esseen rate of convergence.
\end{theorem}

\begin{remark}
For bounds in normal approximation for the number of non-induced copies of a given graph $G$, see \cite[Theorem 3.3]{LRP13} and 
\cite[Chapter 3]{Penrose}.
\end{remark}

\begin{remark}
In \cite{ST23} we can find some more examples of Poisson $U$-statistics in stochastic geometry as well as the model of functionals
of Ornstein-Uhlenbeck-L\'evy processes. In all these examples, concentration inequalities of exponential type are available. We skip the details.
\end{remark}
\medskip

\subsection{Infinte weighted 2-runs}
In our next application we consider the situation of an underlying sequence of Rademacher random variables as in subsection \ref{subRademacher}.
 Let $X = (X_i)_{i \in \Zz}$ be a double sided sequence of \textit{iid} Rademacher random variables such that $\Pp(X_i = 1) = \Pp(X_i = -1) = \frac{1}{2}$ and let for each $n \in \N, (a^{(n)}_i)_{i \in \Zz}$ be a double-sided square-summable sequence of real numbers.
\\ The sequence $(F_n)_{n \in \N}$ of standardized infinite weighted 2-runs is then defined as
\begin{align*}
F_n := \frac{G_n - \mathbb{E}[G_n]}{\sqrt{\Var(G_n)}}, \quad G_n := \sum_{i \in \Zz} a^{(n)}_i \xi_i \xi_{i+1}, \quad n \in \N,
\end{align*}
where $\xi_i := \frac{X_i + 1}{2}$ for $i \in \Zz$. Since it is often nice to work with centered random variables we rewrite $F_n$ as
\begin{align*}
	F := F_n = \frac{1}{\sqrt{\Var(G_n)}} \sum_{i \in \Zz} a^{(n)}_i \left[\xi_i \xi_{i+1} - \frac{1}{4}\right] = \frac{1}{4 \sqrt{\Var(G_n)}} \sum_{i \in \Zz} a^{(n)}_i \left[X_i  + X_i X_{i+1} + X_{i+1}\right].
\end{align*}
We recall further $\norm{a^{(n)}}_{l^p(\Zz)} := \prth{\sum_{i \in \Zz} \abs{a_i}^p}^{1/p} < \infty \, \forall p \geq 2$ as $(a^{(n)}_i)_{i \in \Zz} \in l^2(\Zz)$, as well as $\mathbb{E}[X^k] = 1$ for $k$ even and $\mathbb{E}[X^k] = 0$ for $k$ odd, and $\Var(G_n)  = O\prth{\norm{a^{(n)}}_{l^2(\Zz)}^2}$. We want to apply Theorem~\ref{Theorem:NU_Rademacher} for $k=3$, so we have to show that
\begin{align}\label{SixthMomentTwoRuns}
	\mathbb{E}[F_n^6] = \frac{1}{4^6 (\Var(G_n))^3} \sum_{(i,j,k,l,m,r) \in \Zz^6} a^{(n)}_i a^{(n)}_j a^{(n)}_k a^{(n)}_l a^{(n)}_m a^{(n)}_r \mathbb{E}\prth{A_i\cdot \ldots \cdot A_r} < \infty,
\end{align}
where $A_i = \left[X_i  + X_i X_{i+1} + X_{i+1}\right]$. If we compute $A_i \cdot \ldots \cdot A_r$ for a fixed index $(i,j,k,l,m,r)$, we get $3^6$ summands in total. Among these summands the following cases are possible (apart from the designation of the indices, e.g. $i$ or $i+1$): 
\begin{enumerate}
	\item[(i)] 6 single X's, 0 pairs, so 6 X's in total,
	\item[(ii)] 5 single X's, 1 pair, so 7 X's in total,
	\item[(iii)] 4 single X's, 2 pairs, so 8 X's in total,
	\item[(iv)] 3 single X's, 3 pairs, so 9 X's in total,
	\item[(v)] 2 single X's, 4 pairs, so 10 X's in total,
	\item[(vi)] 1 single X, 5 pairs, so 11 X's in total,
	\item[(vii)] 0 single X's, 6 pairs, so 12 X's in total
\end{enumerate}
and by a \textit{pair} we denote a term of the form $X_i X_{i+1}$.
\\ Case (i): When is $\mathbb{E}[X_i X_j X_k X_l X_m X_r] \neq 0$? Since the odd moments of $X$ are equal to 0 and the random variables are independent this happens if and only if the numbers of equal indices are even, namely
\begin{itemize}
	\item 6 equal indices, e.g. $\mathbb{E}[X_i^6] = 1$,
	\item 4 and 2 equal indices, e.g. $\mathbb{E}[X_i^4] \cdot \mathbb{E}[X_i^2] = 1$,
	\item 2 and 2 and 2 equal indices, e.g. $\mathbb{E}[X_i^2] \cdot \mathbb{E}[X_j^2] \cdot \mathbb{E}[X_k^2] = 1$.
\end{itemize}
By our usual argumentation with AGM-inequality (the indices of $a^{(n)}$ and $X$ do not have to be exactly the same but up to a natural number), the corresponding subterms can be estimated by
\begin{align*}
	\frac{C \norm{a^{(n)}}_{l^6(\Zz)}^6}{(\Var(G_n))^3} = O(1)
\end{align*}
and
\begin{align*}
     \frac{C \norm{a^{(n)}}_{l^4(\Zz)}^4 \norm{a^{(n)}}_{l^2(\Zz)}^2}{(\Var(G_n))^3} = O(1)
\end{align*}
and
\begin{align*}
	\frac{C \norm{a^{(n)}}_{l^2(\Zz)}^6}{(\Var(G_n))^3} = O(1),
\end{align*}
where we also used $\Var(G_n)  = O\prth{\norm{a^{(n)}}_{l^2(\Zz)}^2}$ and $\norm{a^{(n)}}_{l^m(\Zz)}^m \leq \norm{a^{(n)}}_{l^{m-1}(\Zz)}^{m-1} \cdot C \cdot \Var(G_n))^{1/2}$ from section 4 in \cite{BER23}.
\\ Case (ii): Since $i \neq i+1$ it is $\mathbb{E}[X_i X_{i+1}X_j X_k X_l X_m X_r] = 0$ for any index combination of this type.
\\ Case (iii): Since $i \neq i+1$ and $j \neq j+1$ there can not be more than six equal indices in $\mathbb{E}[X_i X_{i+1}X_j X_{j+1}X_k X_l X_m X_r]$ for any index combination of this type. In other words, equal indices are from (up to six) different index sets. Note further that although the number of $X's$ increases with every case, the number of ${a^{(n)}}'s$ always stays at six. So up to a certain degree we have a freedom of choice how we construct our $\norm{a^{(n)}}_{l^m(\Zz)}$ but $m \leq 6$. Basically our bound will be a mixture of the norms appearing in case (i).
\\ Cases (iv) and (vi): analogue to (ii)
\\ Cases (v) and (vii): analogue to (iii)
\\ Putting all the cases together we have shown \eqref{SixthMomentTwoRuns}. We refer to the uniform bound in \cite[Theorem~1.1]{ERTZ22} to write down our result:
\begin{theorem}[Non-uniform Berry--Esseen bound for 2-runs]\label{Theorem:NU_TwoRuns}
	In the setting of infinite weighted 2-runs from above we have
	\begin{align*}
		\abs{\Pp(F \leq z) - \Phi(z)} \leq \frac{C}{\prth{1 + \abs{z}}^3} \prth{\frac{\norm{a^{(n)}}_{l^4(\Zz)}^2}{\norm{a^{(n)}}_{l^2(\Zz)}^2} },
	\end{align*}
and $C$ is a constant depending on the coefficient sequence $(a^{(n)}_i)_{i \in \Zz}$.
\end{theorem}
\medskip

\subsection{Subgraph counts in the Erd\H{o}s--R\'{e}nyi random graph}
In our last application we consider again a sequence of Rademacher random variables as in subsection \ref{subRademacher}.

We start with the complete graph on $n$ vertices and keep an edge with probability $p \in [0,1]$, while we remove it with probability $q := 1-p$, for all edges independently from each other. The outcome is known as the classical Erd\H{o}s--R\'{e}nyi random graph \textbf{G}(n,p) and in many applications $p$ depends on $n$. We fix a graph $G_0$ with at least one edge and consider the number $W$ of subgraphs of \textbf{G}(n,p), which are isomorphic to $G_0$. The corresponding standardized random variable is then defined as 
\begin{align*}
F & := \frac{W - \mathbb{E}[W]}{\sqrt{\Var(W)}}.
\end{align*}
For our result we have to define the important quantity
\begin{align*}
\Psi := \min_{\substack{H \subset G_0 \\  e_H \geq 1}} \left \{ n^{v_H} p^{e_H} \right \},
\end{align*}
where $v_H$ denotes the number of vertices of a subgraph $H$ of $G_0$ and $e_H$ the number of edges, respectively. In  \cite[Theorem~2]{R88} a central limit theorem for $F$ was shown and the core of the proof was the method of moments. The idea of this method is to show that the moments of $F$ converge to the moments of the standard normal distribution, which is unique determined by its moments. So, in particular all moments of $F$ are bounded and we can apply Theorem~\ref{Theorem:NU_Rademacher}. We refer to the uniform bound in \cite[Theorem~1.2]{ERTZ22} to write down our result:
\begin{theorem}[Non-uniform Berry--Esseen bound for subgraph counts]\label{Theorem:NU_ER}
	In the setting of subgraph counts in the Erd\H{o}s--R\'{e}nyi random graph from above we have
	\begin{align*}
		\abs{\Pp(F \leq z) - \Phi(z)} \leq \frac{C}{\prth{1 + \abs{z}}^k} O\prth{(q \Psi)^{-\frac{1}{2}}},
	\end{align*}
and $C$ is a constant depending on $k \in \N$ and on $G_0$.
\end{theorem}

\begin{remark}
In \cite{ERTZ22}, some more applications were considered. The number of vertices with prescribed degree in the Erd\H{o}s-R\'enyi random graph, the number
of vertices of fixed degree for percolation on the Hamming hypercube, the number of isolated faces in the so-called Linial-Meshulam-Wallach random $\kappa$-complex
are some more applications, were we can obtain a non-uniform Berry-Esseen bound as well. We skip the details.
\end{remark}

\end{document}